\documentclass[a4paper,12pt]{amsart}

\usepackage{amsmath}
\usepackage{amsmath}
\usepackage{amssymb}
\usepackage{amscd}
\usepackage{mathrsfs}
\usepackage[all]{xy}
\usepackage{yhmath}

\def\mathcal{\mathscr}
\everymath{\displaystyle}
\newtheorem{thm}{Theorem}[section]
\newtheorem{lem}[thm]{Lemma}
\newtheorem{cor}[thm]{Corollary}
\newtheorem{prop}[thm]{Proposition}

\theoremstyle{definition}

\newtheorem{rem}[thm]{Remark}

\newtheorem{que}[thm]{Question}

\newcommand{\mca}[1]{{\mathcal{#1}}}

\def\Q{{\mathbb Q}}
\def\Z{{\mathbb Z}}
\def\C{{\mathbb C}}
\def\R{{\mathbb R}}

\def\diam{\text{\rm diam}}

\def\ep{\varepsilon} 

\def\ech{\text{\rm ECH}}
\def\emb{\text{\rm emb}}

\def\inj{\text{\rm inj}\,}

\def\ph{\varphi}

\def\supp{\text{\rm supp}\,}
\def\vol{\text{\rm vol}}

\begin{document}
\pagestyle{plain}
\thispagestyle{plain}

\title[Equidistributed periodic Reeb orbits]
{Equidistributed periodic orbits of $C^\infty$-generic three-dimensional Reeb flows}

\author[Kei Irie]{Kei Irie}
\address{Graduate School of Mathematical Sciences, The University of Tokyo, 3-8-1 Komaba, Meguro-Ku, Tokyo 153-8914, Japan}
\email{iriek@ms.u-tokyo.ac.jp} 
\date{\today}

\begin{abstract} 
We prove that, for a $C^\infty$-generic contact form $\lambda$ adapted to a given contact distribution on a closed three-manifold,
there exists a sequence of periodic Reeb orbits which is equidistributed with respect to $d\lambda$. 
This is a quantitative refinement of the $C^\infty$-generic density theorem for three-dimensional Reeb flows, which was previously proved by the author. 
The proof is based on the volume theorem in embedded contact homology (ECH) by Cristofaro-Gardiner, Hutchings, Ramos, 
and inspired by the argument of Marques-Neves-Song, who proved a similar equidistribution result for minimal hypersurfaces. 
We also discuss a question about generic behavior of periodic Reeb orbits ``representing'' ECH homology classes, 
and give a partial affirmative answer to a toy model version of this question which concerns boundaries of star-shaped toric domains. 
\end{abstract}

\maketitle

\section{Introduction} 

\subsection{Setting} 

Let $Y$ be a closed $C^\infty$-manifold of dimension $3$, 
and $\xi$ be a contact distribution on $Y$. 
Namely, $\xi$ is an oriented plane field on $Y$, 
such that there exists a $1$-form $\lambda$ on $Y$ 
satisfying 
\begin{equation}\label{180421_3} 
\ker \lambda = \xi, \qquad 
d\lambda|_{\xi}>0.
\end{equation} 
$Y$ is oriented so that $\lambda \wedge d \lambda$ is positive, 
and we denote
$\vol(Y, \lambda):= \int_Y \lambda \wedge d\lambda$. 
For each positive integer $l$, 
let $\Lambda_{C^l}(Y, \xi)$ denote the set of $C^l$-class $1$-forms satisfying (\ref{180421_3}). 
We abbreviate $\Lambda_{C^\infty}(Y,\xi)$ as
$\Lambda(Y, \xi)$. 

\textbf{Metric and topology on $\Lambda(Y, \xi)$} 

To define a metric and topology on $\Lambda(Y, \xi)$, 
we fix an auxiliary Riemannian metric on $Y$, 
and define
$\| \cdot \|_{C^\infty} : C^\infty(Y, \R) \to \R_{\ge 0}$ by 
\[ 
\| f\|_{C^\infty}  := \sum_{k=0}^\infty 2^{-k} \frac{ \| \nabla^k f\|_{C^0}}{ 1 + \| \nabla^k f\|_{C^0}}. 
\] 
For any $\lambda, \lambda' \in \Lambda(Y, \xi)$, 
one can define $f \in C^\infty(Y, \R)$ 
by $\lambda' = e^f \lambda$.
Then we define a metric $d_{C^\infty}$ on $\Lambda(Y, \xi)$ by 
\[ 
d_{C^\infty}(\lambda, \lambda') := \| f \|_{C^\infty}. 
\] 
This metric induces the usual $C^\infty$-topology on $\Lambda(Y, \xi)$.

\textbf{Reeb orbits and currents} 

For each $\lambda \in \Lambda(Y, \xi)$, 
the \textit{Reeb vector field} $R_\lambda \in \mca{X}(Y)$ is defined by equations 
\[ 
d\lambda(R_\lambda, \,  \cdot \,) = 0, \qquad \lambda(R_\lambda) = 1. 
\] 
Then we define 
\begin{align*} 
\mca{P}(Y, \lambda)&:= \{ \gamma: \R/T_\gamma \Z \to Y \mid T_\gamma \in \R_{>0}, \, \dot{\gamma} = R_\lambda(\gamma) \},  \\
\mca{P}_\inj(Y, \lambda)&:= \{ \gamma \in \mca{P}(Y, \lambda) \,\, \text{which is injective} \}, \\ 
\mca{P}_\emb(Y, \lambda)&:= \{ \mathrm{Im}(\gamma) \mid \gamma \in \mca{P}(Y, \lambda)\}. 
\end{align*} 
For each $\gamma \in \mca{P}(Y, \lambda)$, 
let $\rho(\gamma, \lambda) \in \mathrm{Aut}(\xi_{\gamma(0)})$ 
denote the linearized return map along $\gamma$ 
of the flow generated by $R_\lambda$. 
$\gamma$ is called \textit{nondegenerate} if $1$ is not an eigenvalue of $\rho(\gamma, \lambda)$. 
$\lambda$ is called \textit{nondegenerate} if all elements of $\mca{P}(Y, \lambda)$ (including multiple orbits) are nondegenerate.
Each $\gamma \in \mca{P}_{\emb}(Y, \lambda)$ is oriented so that $R_\lambda$ is of positive direction.
$T_\gamma:= \int_\gamma \lambda$ is called the \textit{period} of $\gamma$. 

A \textit{positive Reeb current} of $(Y, \lambda)$  is a $1$-dimensional current $C$ on $Y$ of the form 
\[ 
C = \sum_{1 \le i \le k} a_i \gamma_i 
\] 
where $a_1, \ldots, a_k$ are positive real numbers, 
and $\gamma_1, \ldots, \gamma_k$ are distinct elements of $\mca{P}_\emb(Y, \lambda)$. 
In other words, 
\[ 
C(\alpha) = \sum_{1 \le i \le k} a_i \int_{\gamma_i} \alpha \qquad(\forall \alpha \in \Omega^1(Y))
\] 
where $\Omega^1(Y)$ denotes the space of all $C^\infty$-class $1$-forms on $Y$. 
Note that $a_1, \ldots, a_k, \gamma_1, \ldots, \gamma_k$ are uniquely determined from $C$, 
up to permutations. 
$C$ is called \textit{nondegenerate} if each $\gamma_i \,(1 \le i \le k)$ is nondegenerate 
(strictly speaking, it means that $\gamma_i$ is an image of a nondegenrate element of $\mca{P}_\inj(Y, \lambda)$). 

Let $\mca{C}(Y, \lambda)$ denote the set of all positive Reeb currents of $(Y, \lambda)$. 
We also define 
\[
\mca{C}_\Z(Y, \lambda) := \{ \sum_{1 \le i \le k}  a_i \gamma_i  \in \mca{C}(Y, \lambda) \mid a_1, \ldots, a_k \in \Z_{>0}  \}. 
\] 

\subsection{Main Result} 

Let us state the main result. 
Let $Y$ be a closed three-manifold, 
$\xi$ be a contact distribution on $Y$, 
and $\Lambda(Y, \xi)$ be the set of 
$C^\infty$-class $1$-forms satisfying (\ref{180421_3}), 
equipped with the $C^\infty$-topology. 

On any topological space $X$, 
we say that a certain property holds for 
\textit{generic} $x \in X$
if the set of all $x \in X$ satisfying this property is residual, 
i.e. it contains an intersection of countably many open and dense sets. 

\begin{thm}\label{180110_1} 
For generic $\lambda \in \Lambda(Y, \xi)$, 
there exists a sequence $(C_k)_{k \ge 1}$ in 
$\mca{C}(Y, \lambda)$ which weakly converges to $d\lambda$, namely 
\begin{equation}\label{180110_3} 
\lim_{k \to \infty} C_k(\alpha) = \int_Y \alpha \wedge d\lambda \qquad ( \forall \alpha \in \Omega^1(Y)). 
\end{equation} 
\end{thm} 

\begin{rem}\label{180110_2}
(\ref{180110_3}) is equivalent to 
\begin{equation}\label{181017_1}
\lim_{k \to \infty} C_k (f\lambda ) = \int_Y f \lambda \wedge d\lambda \qquad(\forall f \in C^\infty(Y, \R)). 
\end{equation} 
Indeed, any $\alpha \in \Omega^1(Y)$ can be written as 
$\alpha = \beta + f\lambda$ by some $f \in C^\infty(Y, \R)$ so that 
$\beta(R_\lambda) \equiv 0$, which implies $C(\beta)=0$ for any positive Reeb current $C$, 
and $\beta \wedge d \lambda \equiv 0$. 
Then, (\ref{181017_1}) implies 
\[ 
\lim_{k \to \infty} C_k(\alpha) = \lim_{k \to \infty} C_k(f\lambda) = \int_Y f\lambda \wedge d \lambda = \int_Y \alpha \wedge d\lambda. 
\] 
\end{rem} 

Theorem \ref{180110_1} is a quantitative refinement of 
the $C^\infty$-generic density theorem for three-dimensional Reeb flows,
which was previously proved in \cite{Irie}: 

\begin{cor}[\cite{Irie}]
For generic $\lambda \in \Lambda(Y, \xi)$, 
the union of periodic orbits of $R_\lambda$ is dense in $Y$. 
\end{cor} 

As noted in \cite{Irie}, 
the $C^2$-version of this result follows from the Hamiltonian $C^1$-closing lemma by Pugh-Robinson \cite{PR}. 
However, ``Hamiltonian $C^\infty$-closing lemma'' is known to be false by Herman \cite{Hermann}. 

As another corollary we get the following result, 
which looks closer to the equidistribution result for minimal hypersurfaces 
by Marques-Neves-Song \cite{MNS}:

\begin{cor} 
For generic $\lambda \in \Lambda(Y, \xi)$, 
there exists a sequence $(\gamma_k)_{k \ge 1}$ in $\mca{P}_\emb(Y, \lambda)$ such that 
\[ 
\lim_{k \to \infty} \frac{\gamma_1 + \ldots + \gamma_k}{T_{\gamma_1} + \cdots + T_{\gamma_k}} = \frac{d\lambda}{\vol(Y, \lambda)}
\] 
weakly as currents. 
\end{cor}
\begin{rem} 
We do not assume $i \ne j \implies \gamma_i \ne \gamma_j$. 
\end{rem} 
\begin{proof} 
Let $(C_k)_{k \ge 1}$ be a sequence in $\mca{C}(Y, \lambda)$ which weakly converges to $\frac{d\lambda}{\vol (Y, \lambda)}$ as $k \to \infty$. 
Each $C_k$ can be written as
\[ 
C_k = \sum_{j=1}^{m_k} a_{k, j} \gamma_{k,j} 
\]
where each $a_{k,j}$ is a positive real number, and 
$\gamma_{k,1}, \ldots, \gamma_{k, m_k}$ 
are distinct elements in $\mca{P}_\emb(Y, \lambda)$. 
Let us take $a'_{k,j} \in \Q_{>0}$ such that 
\[ 
\lim_{k \to \infty} \sum_{j=1}^{m_k}  \lvert a_{k,j} - a'_{k,j} \rvert \cdot T_{\gamma_{k,j}} = 0. 
\] 
Setting 
$C'_k := \sum_{j=1}^{m_k} a'_{k,j} \gamma_{k,j}$ 
for each $k$, 
the sequence $(C'_k)_k$ 
weakly converges to $\frac{d\lambda}{\vol(Y, \lambda)}$ as $k \to \infty$. 
Let us write 
\[ 
C'_k = \frac{1}{n_k} \sum_{j=1}^{m_k} l_{k,j} \gamma_{k,j}
\]
where $n_k$ and $l_{k,j}$ are positive integers. 
Since $\lim_{k \to \infty} C'_k(\lambda)=1$, we obtain
\[ 
\lim_{k \to \infty} \frac{1}{n_k} \sum_{j=1}^{m_k} l_{k,j} T_{\gamma_{k,j}} = 1, 
\] 
thus 
\[ 
\lim_{k \to \infty} \frac{\sum_{j=1}^{m_k} l_{k,j} \gamma_{k,j}}{\sum_{j=1}^{m_k} l_{k,j} T_{\gamma_{k,j}}} = \frac{d\lambda}{\vol(Y, \lambda)}. 
\] 
Now we can complete the proof with 
Lemma \ref{180123_1} below. 
\end{proof}

\begin{lem}\label{180123_1}
Let $(a_n)_{n \ge 1}$ be a sequence of positive real numbers such that $\inf_n a_n >0$, 
and $(n_k)_{k \ge 1}$ be a strictly increasing sequence of positive integers. 
Then there exists a sequence $(m_k)_{k \ge 1}$ of positive integers 
which satisfies the following property: 
\begin{quote}
Let $(b_n)_{n \ge 1}$ be a sequence of real numbers and $\alpha \in \R$ such that 
\[
\sup_n \frac{|b_n|}{a_n} < \infty, 
\qquad 
\lim_{k \to \infty} \frac{b_{n_{k-1}+1} + \cdots + b_{n_k}}{a_{n_{k-1}+1} + \cdots + a_{n_k}} = \alpha.
\] 
Then $\lim_{k \to \infty} \frac{b_{m_1} + b_{m_2} + \cdots + b_{m_k}}{a_{m_1} + a_{m_2} + \cdots + a_{m_k}} = \alpha$. 
\end{quote} 
\end{lem} 
\begin{proof}
Let us define a sequence $(d_k)_{k \ge 1}$ so that $d_1:= n_1$ and $d_k:= n_k - n_{k-1}$ for $k \ge 2$. 
Take a sequence of positive integers $(r_k)_{k \ge 1}$ so that 
$\lim_{k \to \infty} \frac{a_{n_k+1} + \cdots + a_{n_{k+1}}}{r_k \cdot (a_{n_{k-1}+1} + \cdots + a_{n_k})} = 0$. 
Then we set $n_0:= 0$ and define a sequence $(m_k)_{k \ge 1}$ by 
\[ 
m_{r_1 d_1 + \cdots + r_{l-1}d_{l-1} + sd_l + p} := n_{l-1} + p \quad( l \ge 1, \, 0 \le s \le r_l - 1, \, 1 \le p \le d_l).
\]
It is easy to check that the sequence $(m_k)_{k \ge 1}$ satisfies the required property. 
\end{proof} 

\subsection{Generic behavior of periodic Reeb orbits ``representing'' ECH homology classes}

It seems natural to expect that, 
for a $C^\infty$-generic contact form, 
positive Reeb currents ``representing'' ECH homology classes become equidistributed
as grading of the ECH homology classes goes to $\infty$.
Here we formulate this idea as follows, assuming the content of Section 2. 

For any $\lambda \in \Lambda(Y, \xi)$ and $\sigma \in \ech(Y, \xi, \Gamma) \setminus \{0\}$, 
there exists $C \in \mca{C}_{\Z}(Y, \lambda)$ such that
$C(\lambda) = c_\sigma(Y, \lambda)$. 
We say that such $C$ \textit{represents} $\sigma$ with $\lambda$. 
For generic $\lambda$ and every $\sigma \in \ech(Y, \xi, \Gamma)$, 
there exists a unique element of $\mca{C}_\Z(Y, \lambda)$ which represents $\sigma$ with $\lambda$. 
Indeed, generic $\lambda$ satisfies the following property: 
for any distinct elements $\gamma_1, \ldots, \gamma_k \in \mca{P}(Y, \lambda)$, 
their periods $T_{\gamma_1}, \ldots, T_{\gamma_k}$ are linearly independent over $\Q$. 

Now we can formulate the idea explained in the first paragraph as follows. 
Let us take $\Gamma \in H_1(Y: \Z)$ such that $c_1(\xi) + 2 \mathrm{PD}(\Gamma)$ is torsion in $H^2(Y: \Z)$, 
and let $I$ denote the relative $\Z$-grading on $\ech(Y, \xi, \Gamma)$ (see Section 2). 

\begin{que}\label{181218_1} 
Let $(\sigma_k)_{k \ge 1}$ be a sequence 
of nonzero homogeneous elements in $\ech(Y, \xi, \Gamma)$
such that $I(\sigma_{k+1}, \sigma_k) = 2$ for every $k$. 
Then, does the following property hold for a generic element $\lambda$
in $\Lambda(Y, \xi)$ ?
\begin{quote} 
If $(C_k)_{k \ge 1}$ is a sequence of currents on $Y$ such that 
$C_k$ represents $\sigma_k$ with $\lambda$ for every $k \ge 1$, 
then $\lim_{k \to \infty} \frac{C_k}{\sqrt{2k}} = \frac{d\lambda}{\sqrt{\vol(Y, \lambda)}}$. 
\end{quote} 
\end{que} 

Note that both Theorem \ref{180110_1} and the volume theorem in ECH follow from an affirmative answer to Question \ref{181218_1}. 
In Section 6, we formulate a toy model version of Question \ref{181218_1} for boundaries of star-shaped toric domains in $\C^2$, 
and give an affirmative answer for boundaries of strictly convex or concave toric domains. 

\subsection{Structure of this paper} 

The proof of Theorem \ref{180110_1} is based on the volume theorem in embedded contact homology \cite{CGHR}, 
and inspired by the argument in \cite{MNS}. 
The argument in \cite{MNS}, 
which is based on the volume theorem (or Weyl Law) for the volume spectrum \cite{LMN}, 
is a beautiful quantitative refinement of the argument in \cite{IMN}. 

Let us explain the structure of this paper. 
Section 2 collects some facts from the theory of embedded contact homology. 
Section 3 gives a proof of Theorem \ref{180110_1} assuming Lemmas  \ref{Localnondegenerate}, \ref{Familynondegenerate}, \ref{DerivativeOfCapacity}, \ref{MNS}. 
Lemma \ref{MNS} is same as Lemma 3 in \cite{MNS}. 
Lemmas \ref{Localnondegenerate} and \ref{Familynondegenerate} are proved in Section 4, 
and Lemma \ref{DerivativeOfCapacity} is proved in Section 5. 
Finally in Section 6, we discuss a toy model version of Question \ref{181218_1} for boundaries of star-shaped toric domains. 

\textbf{Acknowledgement.}
The author appreciates Chris Gerig for his email which motivated the author to write this paper, 
and for his comments on preliminary versions of this paper. 
This research is supported by JSPS KAKENHI
Grant Number 18K13407. 

\section{Preliminaries from embedded contact homology} 

In this section we briefly collect some facts from the theory of embedded contact homology (ECH). 
For further details, see \cite{Lecturenote} and references therein. 

Let $Y$ be any closed oriented three-manifold, and 
$\xi$ be any contact distribution on $Y$. 
For any $\Gamma \in H_1(Y: \Z)$, 
one can define a $\Z/2$-vector space
$\ech(Y, \xi, \Gamma)$ 
with a relative $\Z/d$-grading, 
where $d$ denotes the divisivility of 
$c_1(\xi) + 2 \mathrm{PD}(\Gamma)$ in $H^2(Y: \Z)$ mod torsion. 
Here $c_1(\xi)$ denotes the first Chern class of $\xi$ equipped with a complex structure $J$
such that $d\lambda(v, Jv)>0$ for any $v \in \xi_p \setminus \{0\} \, (\forall p \in Y)$. 
In particular, if $c_1(\xi) + 2 \mathrm{PD}(\Gamma)$ is torsion in $H^2(Y: \Z)$, 
then $\ech(Y, \xi, \Gamma)$ has a relative $\Z$-grading. 
Note that such $\Gamma$ exists, 
because the fact that $Y$ is parallelizable implies that 
$c_1(\xi) \in 2 H^2(Y: \Z)$. 
We fix such $\Gamma$ in the rest of this paper. 

For any $\sigma \in \ech(Y, \xi, \Gamma) \setminus \{0\}$
and $\lambda \in \Lambda(Y, \xi)$, 
one can define a \textit{spectral invariant}
$c_\sigma(Y, \lambda) \in \R_{\ge 0}$, 
which was introduced in \cite{Quantitative}. 
The spectral invariant satisfies the following properties: 

\textbf{Spectrality.} 
For any $\sigma \in \ech(Y, \xi, \Gamma) \setminus \{0\}$ and $\lambda \in \Lambda(Y, \xi)$
such that $c_\sigma(Y, \lambda)>0$, 
there exist positive integers $a_1, \ldots, a_k$ and $\gamma_1, \ldots, \gamma_k \in \mca{P}_\emb(Y, \lambda)$ such that 
\[ 
c_\sigma(Y, \lambda) = \sum_{1 \le j \le k} a_j T_{\gamma_j}. 
\] 

\textbf{Conformality.} 
$c_\sigma(Y, a\lambda) = a c_\sigma(Y, \lambda)$ for any $a \in \R_{>0}$. 

\textbf{Monotonicity.} 
$c_\sigma(Y, \lambda) \le c_\sigma(Y, f\lambda)$ for any $f \in C^\infty(Y, \R_{\ge 1})$. 

\textbf{$C^0$-continuity.}
Let $(f_j)_{j \ge 1}$ be a sequence in $C^\infty(Y, \R_{>0})$ such that $\lim_{j \to \infty} \| f_j - 1 \|_{C^0}=0$. 
Then $\lim_{j \to \infty} c_\sigma(Y, f_j \lambda) = c_\sigma(Y, \lambda)$. 

\textbf{Volume theorem.} 
Assume that $Y$ is connected, and 
let $(\sigma_k)_{k \ge 1}$ be a sequence of nonzero homogeneous elements in $\ech(Y, \xi, \Gamma)$ such that 
$I(\sigma_{k+1}, \sigma_k) = 2$ for any $k$, where $I$ denotes the relative grading. 
Then, for any $\lambda \in \Lambda(Y, \xi)$ there holds
\[
\lim_{k \to \infty} \frac{c_{\sigma_k}(Y, \lambda)}{\sqrt{k}} = \sqrt{2 \vol(Y, \lambda)}. 
\] 

Here are some explanations on these properties. 
Spectrality follows easily from the definition of spectral invariant; see \cite{Irie} Lemma 2.4. 
Conformality is straightforward from the definition, 
and monotonicity follows from cobordism maps between filtered ECH; 
see \cite{Quantitative}. 
$C^0$-continuity is an immediate consequence of conformality and monotonicity. 
Volume theorem is proved in \cite{CGHR}. 
Note that there always exists a sequence $(\sigma_k)_{k \ge 1}$ which satisfies the assumption of the volume theorem. 
This follows from the isomorphism between ECH and a version of Seiberg-Witten Floer cohomology by \cite{Taubes} and subsequent papers, 
and the corresponding existence result in Seiberg-Witten Floer cohomology by \cite{KM}. 

\section{Proof of Theorem \ref{180110_1} modulo lemmas} 

In Section 3.1, we state Lemmas \ref{Localnondegenerate}, \ref{Familynondegenerate}, \ref{DerivativeOfCapacity}, \ref{MNS}.
In Section 3.2, we reduce Theorem \ref{180110_1} to the key statement Proposition \ref{180110_7}, which we prove in Section 3.3
assuming these lemmas. 

\subsection{Statements of lemmas} 

\begin{lem}\label{Localnondegenerate} 
Let $\lambda \in \Lambda(Y, \xi)$, 
and $\gamma_1, \ldots, \gamma_k$ be distinct elements in $\mca{P}_\emb(Y, \lambda)$. 
For any $\ep>0$, there exists $\lambda' \in \Lambda(Y, \xi)$ such that 
\begin{itemize} 
\item 
$d_{C^\infty}(\lambda, \lambda') < \ep$. 
\item
$\gamma_1, \ldots, \gamma_k \in \mca{P}_\emb(Y, \lambda')$.
\item 
$\gamma_1, \ldots, \gamma_k$ are nondegenerate with respect to $\lambda'$. 
\end{itemize} 
\end{lem} 
Lemma \ref{Localnondegenerate} is an immediate consequence of 
Lemma \ref{180422_1} which we state later. 

Let $N \in \Z_{>0}$, 
and $\lambda$ be a $C^\infty$-section of the bundle 
$\pi^*(T^*Y)$, where $\pi: Y \times [0,1]^N \to Y$ is the projection map. 
For each $\tau \in [0,1]^N$ we define $\lambda_\tau \in \Omega^1(Y)$ by 
$\lambda_\tau(y):= \lambda(\tau, y)$. 
Let $\Lambda^N(Y, \xi)$ denote the set of  $\lambda$ such that 
$\lambda_\tau \in \Lambda(Y, \xi)$ for any $\tau \in  [0,1]^N$. 
The set $\Lambda^N(Y,\xi)$ is equipped with the topology induced from the 
$C^\infty$-topology on the space of $C^\infty$-sections of $\pi^*(T^*Y)$. 

\begin{lem}\label{Familynondegenerate} 
For generic $\lambda \in \Lambda^N(Y, \xi)$, 
\[ 
\mathrm{measure}(\{ \tau \in [0,1]^N \mid \text{$\lambda_\tau$ is nondegenerate} \} ) = 1. 
\] 
\end{lem} 
\begin{rem}
This lemma is a family version of the well-known fact that 
a generic element in $\Lambda(Y, \xi)$ is nondegenerate. 
\end{rem}

The next lemma computes derivatives of ECH spectral invariants
under perturbations of contact forms. 
This is analogous to part of Lemma 2 in \cite{MNS}. 

\begin{lem}\label{DerivativeOfCapacity}

Let $\lambda \in \Lambda^N(Y, \xi)$, 
$\tau^0 \in (0,1)^N$ and 
$\sigma \in \ech(Y, \xi, \Gamma) \setminus \{0\}$. 
We assume that $\lambda_{\tau^0}$ is nondegenerate 
and the function 
\[ 
[0,1]^N \to \R; \, \tau = (\tau_1, \ldots, \tau_N)  \mapsto c_\sigma(\lambda_\tau)
\]
is differentiable at $\tau^0$. 
Then, there exists $C \in \mca{C}_\Z(Y, \lambda_{\tau^0})$ such that 
\begin{align*} 
C(\lambda_{\tau^0}) &= c_\sigma(Y, \lambda_{\tau^0}),  \\ 
C(\partial_{\tau_i} \lambda_\tau(\tau^0)) &= \partial_{\tau_i} (c_\sigma(Y, \lambda_\tau)) (\tau^0)\quad(\forall i \in \{1, \ldots, N\}). 
\end{align*} 
\end{lem} 

The next lemma is exactly the same as Lemma 3 in \cite{MNS}, 
though our notations are slightly different from \cite{MNS}. 

\begin{lem}[\cite{MNS}]\label{MNS}
For any $\delta \in \R_{>0}$ and $N \in \Z_{>0}$, 
there exists $\ep \in \R_{>0}$ depending on $\delta$ and $N$, 
such that the following statement holds true: 

For any Lipschitz function $f: [0,1]^N \to \R$ with $\max f - \min f \le 2\ep$ and 
and a full measure subset $\mca{A} \subset [0,1]^N$, 
there exist $N+1$ sequences 
$(\tau^{1,m})_{m \ge 1}, \ldots, (\tau^{N+1, m})_{m \ge 1}$
on $\mca{A}$ 
satisfying the following conditions: 
\begin{itemize} 
\item There exists $\tau^\infty  \in (0,1)^N$ such that $\lim_{m \to \infty} \tau^{j, m}  = \tau^\infty$ for any $j \in \{1, \ldots, N+1\}$. 
\item $f$ is differentiable at $\tau^{j, m}$ for any $j \in \{1, \ldots, N+1\}$ and $m \ge 1$. 
\item For any $j \in \{1, \ldots, N+1\}$, there exists a limit $v^j: = \lim_{m \to \infty} \nabla f (\tau^{j,m})$.
Moreover, 
\[ 
d_{\R^N} (0, \mathrm{conv} (v^1, \ldots, v^{N+1})) < \delta. 
\] 
\end{itemize} 
\end{lem} 

\subsection{Proof of Theorem \ref{180110_1} assuming Proposition \ref{180110_7}}

Let us take a sequence $(\psi_i)_{i \ge 1}$ in $C^\infty(Y, \R)$
which is $C^0$-dense in $C^\infty(Y, \R)$ 
and $\psi_1 \equiv 1$. 

\begin{prop}\label{180110_7} 
Let $N \in \Z_{>0}$, $\ep \in \R_{>0}$, 
and $\mca{U}$ be any nonempty open set in $\Lambda(Y, \xi)$. 
Then there exist
$\lambda \in \mca{U}$ and $C \in \mca{C}(Y, \lambda)$ such that 
\begin{equation}\label{180110_8} 
\bigg\lvert  C (\psi_i \lambda )- \int_Y \psi_i  \lambda \wedge d\lambda \bigg\rvert < \ep  \qquad(\forall i \in \{1, \ldots, N\}). 
\end{equation} 
\end{prop} 

First we prove Theorem \ref{180110_1} assuming 
Proposition \ref{180110_7}. 
For any $N \in \Z_{>0}$ and $\ep \in \R_{>0}$, 
let $\Lambda(N, \ep)$ denote 
the set of $\lambda \in \Lambda(Y, \xi) $ such that 
there exists $C \in \mca{C}(Y, \lambda)$ which is nondegenerate and 
satisfies (\ref{180110_8}). 

We show that $\Lambda(N, \ep)$ is open and dense in $\Lambda(Y, \xi)$. 
Denseness follows from Proposition \ref{180110_7} 
and Lemma \ref{Localnondegenerate}. 
To show openness, 
let $\lambda \in \Lambda(N, \ep)$ and take 
$C = \sum_{1 \le j \le k} a_j \gamma_j \in \mca{C}(Y, \lambda)$
which is nondegenerate and satisfies (\ref{180110_8}). 
Then there exists a neighborhood $\mca{U}$ of $\lambda$ in $\Lambda(Y, \xi)$ 
and $\gamma_j(\lambda') \in \mca{P}_\emb(Y, \lambda')$ for any $j \in \{1, \ldots, k\}$ and $\lambda' \in \mca{U}$, 
so that $\gamma_j(\lambda')$ varies smoothly on $\lambda'$ 
and $\gamma_j(\lambda) = \gamma_j$. 
If $\lambda' \in \mca{U}$ is sufficiently close to $\lambda$, 
then $C(\lambda'):= \sum_{1 \le j \le k} a_j \gamma_j(\lambda')$ 
is nondegenerate and satisfies (\ref{180110_8}), 
thus $\lambda' \in \Lambda(N, \ep)$. 
This completes the proof of openness. 

Now let us take any sequence $(\ep_N)_{N \ge 1}$ of positive real numbers which converges to $0$, 
and consider a residual set 
\[ 
\Lambda^* := \bigcap_{N \ge 1} \Lambda(N, \ep_N). 
\] 
We show that, 
for any  $\lambda \in \Lambda^*$, 
there exists a sequence $(C_N)_{N \ge 1}$ of $\mca{C}(Y, \lambda)$
which weakly converges to $d\lambda$. 
By $\lambda \in \Lambda(N, \ep_N)$,
there exists $C_N \in \mca{C}(Y, \lambda)$ such that 
\[ 
\bigg\lvert C_N(\psi_i  \lambda) - \int_Y \psi_i  \lambda \wedge d\lambda \bigg\rvert < \ep_N \quad (\forall i \in \{1, \ldots, N\}), 
\] 
in particular $|C_N(\lambda) - \vol(Y, \lambda)| < \ep_N$ since $\psi_1 \equiv 1$. 

Then for any $f \in C^\infty(Y, \R)$ and $i \in \{1, \ldots, N\}$, there holds 
\begin{align*}
&\qquad \bigg\lvert C_N(f\lambda) - \int_Y f\lambda \wedge d\lambda \bigg\rvert \\ 
& \le | C_N(f\lambda) - C_N(\psi_i \lambda)  | 
+ \bigg\lvert C_N(\psi_i \lambda) - \int_Y \psi_i \lambda \wedge d\lambda \bigg\rvert 
+ \bigg\lvert \int_Y (f-\psi_i) \lambda \wedge d\lambda \bigg\rvert \\
&<  \| f - \psi_i \|_{C^0}  (C_N(\lambda) + \vol(Y, \lambda))  + \ep_N \\
&<  \| f - \psi_i \|_{C^0} (2 \vol(Y, \lambda) + \ep_N) + \ep_N. 
\end{align*}
The last inequality follows from 
$|C_N(\lambda) - \vol(Y, \lambda)| < \ep_N$. 
Thus we obtain 
\[ 
\bigg\lvert C_N(f\lambda) - \int_Y f\lambda \wedge d\lambda \bigg\rvert 
< \min_{1 \le i \le N} \| f - \psi_i \|_{C^0}  \cdot  (2 \vol(Y, \lambda) + \ep_N) + \ep_N.
\] 
Since the RHS converges to $0$ as $N \to \infty$, 
we have proved 
$\lim_{N \to \infty} C_N(f\lambda) = \int_Y f \lambda \wedge d\lambda$
for any $f \in C^\infty(Y, \R)$. 
By Remark \ref{180110_2} this shows that $(C_N)_{N \ge 1}$ weakly converges to $d\lambda$. 

\subsection{Proof of Proposition \ref{180110_7}}

The proof consists of four steps. 
In the following argument we assume that $Y$ is connected; 
the general case (i.e. $Y$ may not be connected) easily follows from this case. 

\textbf{Step 1.} 
Take a sequence $(\sigma_k)_{k \ge 1}$ 
of nonzero homogeneous elements in 
$\ech(Y, \xi, \Gamma)$ as in the volume theorem. 

For any $\lambda \in \Lambda(Y, \xi) $ we define 
\[ 
\gamma_k(\lambda) := \frac{c_{\sigma_k}(Y, \lambda)}{\sqrt{k}} - \sqrt{ 2 \vol(Y, \lambda)}. 
\] 
Then the volume theorem implies 
\[ 
\lim_{k \to \infty} \gamma_k(\lambda) = 0
\] 
for any $\lambda \in \Lambda(Y, \xi)$. 

The next Lemma \ref{Lipschitz} 
shows that 
$(\gamma_k)_{k \ge 1}$ is locally uniformly Lipschitz. 

\begin{lem}\label{Lipschitz} 
For any $\lambda_0 \in \Lambda(Y, \xi)$, 
there exists $c \in \R_{>0}$ which depends only on $\lambda_0$, 
and satisfies the following property: 
\begin{quote}
For any $f, f' \in C^\infty(Y, [-1, 1])$ and $k \ge 1$, there holds 
\[ 
| \gamma_k (e^f \lambda_0) - \gamma_k(e^{f'}\lambda_0) | \le c \| f- f' \|_{C^0}. 
\] 
\end{quote} 
\end{lem} 
\begin{proof}
Let us first set
\[ 
\lambda:= e^f \lambda_0, \quad
\lambda':= e^{f'} \lambda_0, \quad
h:= f'- f. 
\] 
Then 
$e^{\min h} \le \lambda'/\lambda \le e^{\max h}$
implies
$e^{2\min h} \le \vol(Y, \lambda')/\vol(Y, \lambda) \le e^{2\max h}$, 
thus 
\begin{align*} 
|\sqrt{\vol(Y, \lambda')} - \sqrt{\vol(Y, \lambda)}| & =  |\vol(Y, \lambda') - \vol(Y, \lambda)| /(\sqrt{\vol(Y, \lambda')} + \sqrt{\vol(Y, \lambda)}) \\
& \le (e^{2 \|h\|_{C^0}} - 1) \sqrt{\vol(Y, \lambda)} \\
& \le e (e^{2 \|h\|_{C^0}} - 1) \sqrt{\vol(Y, \lambda_0)}. 
\end{align*} 
On the other hand, 
by conformality and monotonicity of ECH spectral invariants, 
\[ 
|c_{\sigma_k}(Y, \lambda') - c_{\sigma_k}(Y, \lambda)| \le e (e^{\|h\|_{C^0}} - 1) c_{\sigma_k}(Y, \lambda_0)
\] 
for any $k \ge 1$. 

Hence we obtain 
$|\gamma_k(\lambda) - \gamma_k(\lambda')| \le c \|h\|_{C^0}$, where
\[ 
c: = \frac{e^5-e}{4}  \cdot \bigg(\sup_{k \ge 1} \frac{c_{\sigma_k}(Y, \lambda_0)}{\sqrt{k}} + 2 \sqrt{2 \vol(Y, \lambda_0)} \bigg). 
\] 
\end{proof} 

\textbf{Step 2.} 
We take and fix constants $c_0, \ldots, c_3$. 
We may assume that the diameter of $\mca{U}$ with respect to $d_{C^\infty}$ is finite. 
Then we can take sufficiently large $c_0 \in \R_{>0}$ so that 
\[ 
\max 
\bigg\{
\sup_{k \ge 1} \frac{c_{\sigma_k}(Y, \lambda)}{\sqrt{k}}, \vol(Y, \lambda) 
\bigg\} < c_0 
\] 
for any $\lambda \in \mca{U}$. 

Let us take and fix $\lambda^0 \in \mca{U}$ arbitrarily. 
Then we take sufficiently large $c_1 \in \R_{>0}$ so that 
\[
| \tau \cdot \psi | < c_1, \qquad
e^{(\tau \cdot \psi)/{c_1}} \cdot \lambda^0 \in \mca{U}
\]
for any $\tau = (\tau_1, \ldots, \tau_N) \in [0,1]^N$, 
where $\tau \cdot \psi := \sum_{1 \le j \le N} \tau_j \psi_j$. 
Finally we take sufficiently small $c_2, c_3 \in \R_{>0}$ so that 
\[
\sqrt{\frac{c_0}{2}} \cdot (c_1c_3 + c_2 (c_0 + \sqrt{2c_0})) < \ep.
\]

\textbf{Step 3.} 
Applying Lemma \ref{Familynondegenerate} to $(e^{\tau \cdot \psi/c_1}\lambda^0)_{\tau \in [0,1]^N}$, 
there exists $\lambda = (\lambda_\tau)_{\tau \in [0,1]^N} \in \Lambda^N(Y, \xi)$
satisfying the following conditions: 
\begin{itemize} 
\item $\{ \tau \in [0,1]^N \mid \text{$\lambda_\tau$ is nondegenerate}\}$ is of full measure in $[0,1]^N$. 
\item $\| c_1(\partial_{\tau_i} \lambda_\tau/\lambda_\tau) - \psi_i \|_{C^0} < c_2$ for any $i \in \{1, \ldots, N\}$.  
\end{itemize} 
For any $\tau \in [0,1]^N$, 
let us define $f_\tau \in C^\infty(Y)$ by 
$f_\tau:= \log(\lambda_\tau/\lambda^0)$. 
Then, Lemma \ref{Lipschitz} implies 
\[ 
|\gamma_k(\lambda_\tau) - \gamma_k(\lambda_{\tau'})| \le 
c \| f_\tau - f_{\tau'} \|_{C^0}  \le 
c  \cdot \max_{\substack{1 \le j \le N \\ \tau \in [0,1]^N }} \| \partial_{\tau_j}f_\tau \|_{C^0} \cdot |\tau - \tau'|
\] 
where $c$ is a positive constant which does not depend on $k$. 

Let us define
$\bar{\gamma}_k: [0,1]^N \to \R$ by 
$\bar{\gamma}_k(\tau):= \gamma_k(\lambda_\tau)$. 
Then $(\bar{\gamma}_k)_{k \ge 1}$ is uniformly Lipschitz. 
Since $\lim_{k \to \infty} \bar{\gamma}_k(\tau) = 0$ for any $\tau \in [0,1]^N$
and $(\bar{\gamma}_k)_{k \ge 1}$ is uniformly Lipschitz, 
there holds $\lim_{k \to \infty} \|\bar{\gamma}_k\|_{C^0} = 0$. 
By Lemma \ref{MNS}, 
when $k$ is sufficiently large, 
there exist $N+1$ sequences 
$(\tau^{1,m})_{m \ge 1}, \ldots, (\tau^{N+1,m})_{m \ge 1}$ such that: 
\begin{itemize} 
\item For any $j \in \{1, \ldots, N+1\}$ and $m \ge 1$, 
$\bar{\gamma}_k$ is differentiable at $\tau^{j, m}$, 
and $\lambda_{\tau^{j,m}}$ is nondegenerate. 
\item There exists $\tau_\infty  \in (0,1)^N$ such that $\lim_{m \to \infty} \tau^{j, m}  = \tau^\infty $ for every $j \in \{1, \ldots, N+1\}$. 
\item There exists the limit $\lim_{m \to \infty} \nabla \bar{\gamma}_k (\tau^{j,m}) =:v^j$ for every $j \in \{1, \ldots, N+1\}$, 
and 
\[ 
d_{\R^N} (0, \text{conv}( v^1, \ldots, v^{N+1})) < c_3. 
\] 
\end{itemize}
By the last condition, there exist $a_1, \ldots, a_{N+1} \in [0,1]$ such that $\sum_{1 \le j \le N+1} a_j = 1$ and 
$\bigg\lvert \sum_{1 \le j \le N+1} a_j v^j  \bigg\rvert < c_3$. 
Thus when $m$ is sufficiently large, there holds 
\[ 
\bigg\lvert \sum_{1 \le j \le N+1} a_j  \cdot \partial_{\tau_i} \bar{\gamma_k} (\tau^{j,m}) \bigg\rvert < c_3
 \qquad (\forall i  \in \{1, \ldots, N\}). 
\] 

\textbf{Step 4.}
Now we are going to show that 
there exists $C_\infty \in \mca{C}(Y, \lambda_{\tau^\infty})$ such that 
\[ 
\bigg\lvert C_\infty (\psi_i \lambda_{\tau^\infty} ) - \int_Y \psi_i \lambda_{\tau^\infty} \wedge d\lambda_{\tau^\infty} \bigg\rvert < \ep  \qquad (\forall i  \in \{1, \ldots, N\}), 
\] 
which completes the proof of Proposition \ref{180110_7}. 

By Lemma \ref{DerivativeOfCapacity}, 
for any $j \in \{1, \ldots, N+1\}$ and $m \in \Z_{>0}$, 
there exists $C_{j,m} \in \mca{C}_\Z(Y, \lambda_{\tau^{j,m}})$ such that 
$C_{j,m}(\lambda_{\tau^{j,m}})  = c_{\sigma_k}(\lambda_{\tau^{j,m}})$ and 
\[ 
C_{j,m}(\partial_{\tau_i} \lambda_\tau(\tau^{j,m}) ) = \partial_{\tau_i} c_{\sigma_k}(\lambda_\tau)(\tau^{j,m}) \qquad( \forall i \in \{1, \ldots, N\}). 
\] 
On the other hand, direct computations show that, for any $\tau \in [0,1]^N$ there holds 
\[ 
\partial_{\tau_i} \bar{\gamma}_k(\tau) 
= 
\frac{\partial_{\tau_i} c_{\sigma_k}(\lambda_\tau)}{\sqrt{k}} - 
\sqrt{\frac{2}{\vol(Y, \lambda_\tau)}} \cdot \int_Y \partial_{\tau_i} \lambda_\tau \wedge d \lambda_\tau. 
\] 
Thus we obtain 
\[ 
\bigg\lvert  \sum_{1 \le j \le N+1} a_j \bigg( \frac{C_{j,m}(\partial_{\tau_i} \lambda_\tau(\tau^{j,m}))}{\sqrt{k}} - \sqrt{\frac{2}{\vol(Y, \lambda_{\tau^{j,m}})}} \int_Y \partial_{\tau_i} \lambda_\tau(\tau^{j,m}) \wedge d\lambda_{\tau^{j,m}} \biggr) 
\bigg\rvert  < c_3
\] 
for any $i \in \{1,\ldots,N\}$. 
On the other hand 
\[ 
\| c_1 (\partial_{\tau_i} \lambda_\tau/\lambda_\tau) - \psi_i \|_{C^0} < c_2 \qquad( \forall i \in \{1, \ldots, N\}). 
\] 
Therefore
\begin{align*} 
&\bigg\lvert  \sum_{1 \le j \le N+1} a_j \bigg( \frac{C_{j,m}(\psi_i \lambda_{\tau^{j,m}} )}{\sqrt{k}} - \sqrt{\frac{2}{\vol(Y, \lambda_{\tau^{j,m}})}} \int_Y \psi_i \lambda_{\tau^{j,m}}  \wedge d\lambda_{\tau^{j,m}} \biggr) 
\bigg\rvert  \\
& < c_1 c_3 + c_2 \max_{1 \le j \le N+1}  \bigg( \frac{c_{\sigma_k}(\lambda_{\tau^{j,m}})}{\sqrt{k}} + \sqrt{ 2 \vol (Y, \lambda_{\tau^{j,m}})} \bigg). 
\end{align*} 

For each $1 \le j \le N+1$, 
a certain subsequence of $(C_{j,m})_{m \ge 1}$ 
weakly converges to an element of $\mca{C}_\Z(Y, \lambda_{\tau^\infty})$ as $m \to \infty$. 
Indeed, let 
\[ 
C_{j,m} := \sum_{1 \le i \le I(j,m)} a_{i,j,m} \gamma_{i,j,m} \qquad( a_{i,j,m} \in \Z_{>0}, \, \gamma_{i,j,m} \in \mca{P}_\emb(Y, \lambda_{\tau^{j,m}} ) ). 
\] 
Setting $\delta:= \min \{T_\gamma \mid \gamma \in \mca{P}(Y, \lambda_\tau), \, \tau \in [0,1]^N\}$, 
we obtain 
\[
\sum_{1 \le i \le I(j,m)} a_{i,j,m}  \le  \frac{\max \{ c_{\sigma_k}(\lambda_\tau) \mid \tau \in [0,1]^N \}}{\delta}. 
\]
Thus, up to subsequence, 
we may assume that 
$I(j,m)$ and $a_{i,j,m}\,(1 \le i \le I(j,m))$ 
do not depend on $m$. 
Also, for each $(i, j)$, 
a certain subsequence of $(\gamma_{i,j,m})_{m \ge 1}$ 
weakly converges to $\nu \gamma$ for some $\nu \in \Z_{>0}$ and $\gamma \in \mca{P}_\emb(Y, \lambda_{\tau^\infty})$. 

Hence, up to subsequence, we may assume that 
$\lim_{m \to \infty} \sum_{1 \le j \le N+1} a_j C_{j,m}$
exists as an element of $\mca{C}(Y, \lambda_{\tau^\infty})$. 
Setting 
\[ 
C_\infty:= \sqrt{\frac{\vol(Y, \lambda_{\tau^\infty})}{2k}} \lim_{m \to \infty} \sum_{1 \le j \le N+1} a_j C_{j,m}, 
\] 
we obtain 
\[ 
\bigg\lvert C_\infty(\psi_i \lambda_{\tau^\infty}) - \int_Y \psi_i \lambda_{\tau_\infty} \wedge d \lambda_{\tau^\infty} \bigg\rvert 
< 
\sqrt{\frac{c_0}{2}} \cdot  (c_1c_3 + c_2 (c_0 + \sqrt{2c_0})) < \ep
\]
for any $i \in \{1, \ldots, N\}$. 
This completes the proof. 
\qed

\section{Proof of Lemmas \ref{Localnondegenerate} and \ref{Familynondegenerate}}

In this section, we give a proof of Lemma \ref{Familynondegenerate}, 
and note that Lemma \ref{Localnondegenerate} is a consequence of Lemma \ref{180422_1}, 
which appears in the last part of this section. 

The proof of Lemma \ref{Familynondegenerate} consists of four steps. 

\textbf{Step 1.} 
Proof of Lemma \ref{Familynondegenerate} assuming Lemma \ref{180420_1}. 

Let $\pi: Y \times [0,1]^N \to Y$ denote the projection. 
For each positive integer $l$, 
let $\Lambda^N_{C^l}(Y, \xi)$ denote the set of $C^l$-section $\lambda$ of $\pi^*(TY)$, 
such that $\lambda_\tau:= \lambda(\tau, \cdot \,) \in \Lambda_{C^l}(Y, \xi)$ for every $\tau \in [0,1]^N$. 
The space $\Lambda^N_{C^l}(Y, \xi)$ is equipped with the natural $C^l$-topology. 
Now we prove that Lemma \ref{Familynondegenerate} is reduced to Lemma \ref{180420_1} below. 

\begin{lem}\label{180420_1}
When $l$ is sufficiently large, 
\[
\mathrm{measure} (\{ \tau \in [0,1]^N \mid \text{$\lambda_\tau$ is nondegenerate} \}) = 1 
\] 
for generic $\lambda \in \Lambda^N_{C^l}(Y, \xi)$. 
\end{lem}

Let us prove Lemma \ref{Familynondegenerate} assuming Lemma \ref{180420_1}. 
For any positive real numbers $T$ and $\delta$, 
let $\Lambda(T, \delta)$ denote the set of $\lambda \in \Lambda^N(Y, \xi)$ such that 
\[
\mathrm{measure}(\{ \tau \in [0,1]^N \mid \text{Any $\gamma \in \mca{P}(Y, \lambda_\tau)$ with $T_\gamma \le T$ is nondegenerate} \})  > 1-\delta. 
\] 
Then $\Lambda(T, \delta)$ is open and dense in $\Lambda^N(Y, \xi)$ with the $C^\infty$-topology; 
openness is easy and denseness follows from Lemma \ref{180420_1}. 
Let us take an increasing sequence $(T_n)_{n \ge 1}$ and a decreasing sequence $(\delta_n)_{n \ge 1}$ such that $\lim_{n \to \infty} T_n = \infty$, $\lim_{n \to \infty} \delta_n = 0$. 
If $\lambda$ is in the residual set $\bigcap_{n \ge 1} \Lambda(T_n, \delta_n)$, 
then $\{ \tau \in [0,1]^N \mid \text{$\lambda_\tau$ is nondegenerate} \}$ is of full measure.

\textbf{Step 2.} Proof of Lemma \ref{180420_1} assuming Lemmas \ref{180421_1} and \ref{180421_2}. 

\begin{lem}\label{180421_1} 
Let us define
\[
\mca{M}:= \{(\lambda, \tau, \gamma) \mid \lambda \in \Lambda^N_{C^l}(Y, \xi), \, \tau \in [0,1]^N, \, \gamma \in \mca{P}_\inj(Y, \lambda_\tau) \}. 
\]
Then $\mca{M}$
has a structure of a Banach manifold of class $C^{l-1}$, 
such that the projection map 
$\mca{M} \to \Lambda^N_{C^l}(Y, \xi)$ is a $C^{l-1}$-Fredholm map of index $N+1$. 
\end{lem} 

For any contact from $\lambda$ and $\gamma \in \mca{P}(Y, \lambda)$, 
recall that $\rho(\gamma, \lambda) \in \mathrm{Aut}(\xi_{\gamma(0)})$ denotes the linearized return map 
of the flow generated by $R_\lambda$. 

\begin{lem}\label{180421_2} 
For any $\theta \in (0, 2\pi) \setminus \{\pi\}$, 
\[ 
\mca{M}_\theta: = \{ (\lambda, \tau, \gamma) \in \mca{M} \mid \rho(\gamma, \lambda_\tau) \sim \begin{pmatrix} \cos \theta & - \sin \theta \\ \sin \theta & \cos \theta \end{pmatrix} \} 
\] 
is a submanifold of $\mca{M}$ of codimension $1$. 
Moreover, setting 
\begin{align*} 
\mca{M}_I &:=\{ (\lambda, \tau, \gamma) \in \mca{M} \mid \rho(\gamma, \lambda_\tau)  \sim \begin{pmatrix} 1 & 0 \\ 0 &1 \end{pmatrix} \}, \\
\mca{M}_{II} &:=\{ (\lambda, \tau, \gamma) \in \mca{M} \mid \rho(\gamma, \lambda_\tau)  \sim \begin{pmatrix} 1 & 1 \\ 0 & 1 \end{pmatrix} \},  \\
\mca{M}_{III}&:= \{ (\lambda, \tau, \gamma) \in \mca{M} \mid \rho(\gamma, \lambda_\tau)  \sim \begin{pmatrix} -1 & 0 \\ 0 & -1 \end{pmatrix} \}, \\
\mca{M}_{IV}&:= \{ (\lambda, \tau, \gamma) \in \mca{M} \mid \rho(\gamma, \lambda_\tau)  \sim \begin{pmatrix} -1 & 1 \\ 0 & -1 \end{pmatrix} \}, 
\end{align*} 
$\mca{M}_I$ and $\mca{M}_{III}$ are $C^{l-2}$ submanifolds of $\mca{M}$ of codimension $3$, 
and $\mca{M}_{II}$ and $\mca{M}_{IV}$ are $C^{l-2}$ submanifolds of $\mca{M}$ of codimension $1$. 
\end{lem} 

For any $\lambda \in \Lambda^N_{C^l}(Y, \xi)$, we set 
$\mca{M}(\lambda):= \{ (\tau, \gamma) \mid (\lambda, \tau, \gamma) \in \mca{M}\}$. 
We also define 
$\mca{M}_\theta(\lambda), \mca{M}_I(\lambda), \ldots, \mca{M}_{IV}(\lambda)$ 
in a similar manner. 

Now we prove Lemma \ref{180420_1} assuming Lemmas \ref{180421_1} and \ref{180421_2}. 
By Sard-Smale theorem, 
if $l \ge N+3$, 
a generic element of $\Lambda^N_{C^l}(Y, \xi)$ is a regular value of projection maps from 
$\mca{M}$, $\mca{M}_I, \ldots, \mca{M}_{IV}$ 
and $\mca{M}_{2\pi j/n} \, (n \in \Z_{>2}, \, j \in \{1, \ldots, n-1\} \setminus \{n/2\})$ 
to $\Lambda^N_{C^l}(Y, \xi)$. 
If $\lambda = (\lambda_\tau)_{\tau \in [0,1]^N}$ 
is such a regular value, then 
$\mca{M}(\lambda)$ has a structure of $C^{l-1}$ manifold of dimension $N+1$, 
and
$\mca{M}_I(\lambda), \ldots, \mca{M}_{IV}(\lambda)$ 
and
$\mca{M}_{2\pi j/n}(\lambda)$ are its 
$C^{l-2}$ submanifolds of codimension at least $1$. 

By (finite-dimensional) Sard theorem, 
if $l \ge 3$, 
then the union of critical values of projection maps from 
$\mca{M}_I(\lambda), \ldots, \mca{M}_{IV}(\lambda)$ 
 and $\mca{M}_{2\pi j/n}(\lambda)$ to $[0,1]^N$ 
is a measure zero set. 
If $\tau$ is a regular value of all these maps, 
then actually $\tau$ is not in the image of these maps
(this is because $\mca{M}$ admits a free $S^1$ action by rotating parameters, 
preserving $\mca{M}_I, \ldots, \mca{M}_{IV}, \mca{M}_\theta$ and the projection map to $\Lambda^N_{C^l}(Y, \xi) \times [0,1]^N$)
which means that $\lambda_\tau$ is nondegenerate. 

\textbf{Step 3.} 
Proof of Lemma \ref{180421_1}. 

For any contact form $\lambda$ on $Y$, 
let $(\varphi^t_{R_\lambda})_{t \in \R}$ denote the flow generated by $R_\lambda$. 
Let us consider a map 
\[ 
E: \Lambda^{N}_{C^l}(Y, \xi) \times Y \times [0,1]^N \times \R_{>0} \to Y^2; \quad (\lambda, y, \tau, T) \mapsto (y, \ph^T_{R_{\lambda_\tau}} (y)). 
\] 
It is easy to check that $E$ is of class $C^{l-1}$. 
Let $O$ denote the open set in the source of $E$, which consists of 
$(\lambda, y, \tau, T)$ such that 
\[ 
0 \le \theta < \theta' \le T/2 \implies \ph^\theta_{R_{\lambda_\tau}}(y) \ne \ph^{\theta'}_{R_{\lambda_\tau}}(y). 
\] 
Let $\Delta_Y:= \{(y,y) \mid y \in Y \}$. 
Then $(E|_O)^{-1}(\Delta_Y)$ is naturally identified with $\mca{M}$. 
Moreover, using the formula 
$R_{e^f\lambda} = e^{-f}(R_{\lambda} - X_{df})$
(here $X_{df}$ is a section of $\xi$ defined by $i_{X_{df}}d\lambda = - df$), 
it is easy to check that $E|_O: O \to Y^2$ is transversal to $\Delta_Y$. 
Thus $(E|_O)^{-1}(\Delta_Y)$ is a $C^{l-1}$-submanifold of 
$\Lambda^N_{C^l}(Y, \xi) \times Y \times [0,1]^N \times \R_{>0}$, 
and the projection $(E|_O)^{-1}(\Delta_Y) \to \Lambda^N_{C^l}(Y, \xi)$ is a Fredholm map of index $N+1$. 
This completes the proof of Lemma \ref{180421_1}.

\textbf{Step 4.} 
Proof of  Lemma \ref{180421_2}. 

Let $\mca{E}$ denote the total space of a $\mathrm{Sp}(2: \R)$-bundle on $\mca{M}$ 
defined by 
\[ 
\mca{E}(\lambda, \tau, \gamma):=  \mathrm{Aut}( \xi_{\gamma(0)},  d\lambda_\tau) \qquad (\lambda \in \Lambda^N_{C^l}(Y, \xi), \, \tau \in [0,1]^N, \, \gamma \in \mca{P}_\inj(Y, \lambda_\tau)).
\]
For any $\theta \in (0,2 \pi) \setminus \{\pi\}$, we define a subspace $\mca{E}_\theta$ of $\mca{E}$ by 
\[ 
\mca{E}_\theta (\lambda, \tau, \gamma) := \{ R \in \mathrm{Aut}(\xi_{\gamma(0)}, d\lambda_\tau) \mid R \sim \begin{pmatrix} \cos \theta & - \sin \theta \\ \sin \theta & \cos \theta \end{pmatrix} \}. 
\] 
We also define $\mca{E}_I, \ldots, \mca{E}_{IV} \subset \mca{E}$ in a similar manner. 

Let $s$ be a section of $\mca{E} \to \mca{M}$ defined by 
$s(\lambda, \tau, \gamma):= \rho(\gamma, \lambda_\tau)$. 
Note that $s$ is of class $C^{l-2}$. 
To prove Lemma \ref{180421_2}, it is sufficient to check that $s$ is transversal to $\mca{E}_\theta \,( \theta \in (0, 2\pi) \setminus \{\pi\})$ 
and $\mca{E}_I, \ldots, \mca{E}_{IV}$. 
This follows from Lemma \ref{180422_1} below
and contact Darboux theorem. 

To state Lemma \ref{180422_1}, 
let $\lambda:= x dy + dz$ be the standard contact form on $\R^3$, and 
$\xi$ be the associated contact distribution. 
Also, let $U$ be an open neighborhood of $(0, 0, 0)$ in $\R^3$ such that 
$(0,0,-1)$ and $(0,0,1)$ are not in $U$. 
We set
\[
\mca{F}_U := \{ f \in C^\infty(\R^3) \mid \supp f \subset U, f(0,0,z) = 0, \, df(0,0,z) = 0 \,(\forall z \in \R) \}. 
\] 
For any $f \in \mca{F}_U$, let 
$A(f):\xi_{(0,0,-1)} \to \xi_{(0,0,1)}$ 
denote the linearization of the time-$2$ map 
of the flow generated by the Reeb vector field $R_{e^f \lambda}$. 

\begin{lem}\label{180422_1} 
The map 
\[ 
A: \mca{F}_U \to \{ \varphi \in \mathrm{Hom}(\xi_{(0,0,-1)}, \xi_{(0,0,1)}) \mid \varphi^* d\lambda = d\lambda \}; \quad f \mapsto A(f) 
\]
is submersive at $0 \in \mca{F}_U$. 
\end{lem} 
\begin{proof}
By simple computations, one can check
\[
(dA)_0(h) = \int_{-1}^1 \begin{pmatrix} \partial_{xy}h & \partial_{yy}h \\ -\partial_{xx}h & -\partial_{xy}h  \end{pmatrix}  (0,0,t) \, dt.
\]
Now the conclusion of the lemma easily follows from this formula. 
\end{proof} 

\begin{rem} 
Lemma \ref{Localnondegenerate} 
is an immediate consequence of Lemma \ref{180422_1}. 
\end{rem} 

\section{Proof of Lemma \ref{DerivativeOfCapacity}} 

We first prove the following simple lemma. 

\begin{lem}\label{DerivativeOfAction}
Let $\lambda  \in \Lambda^N(Y, \xi)$, 
$\tau^0 \in (0,1)^N$ and 
$\gamma \in \mca{P}_\emb(\lambda_{\tau^0})$ which is nondegenerate. 
Then there exists a neighborhood $U$ of $\tau^0$ in $(0,1)^N$ 
and a smooth family of embedded Reeb orbits $(\gamma_\tau)_{\tau \in U}$
such that $\gamma_\tau \in \mca{P}_\emb(\lambda_\tau)$ for every $\tau \in U$, 
$\gamma_{\tau^0}=\gamma$, and 
\[ 
\partial_{\tau_i} \bigg( \int_{\gamma_\tau} \lambda_\tau \bigg) (\tau^0) = 
\int_{\gamma} \partial_{\tau_i} \lambda_\tau (\tau^0) \qquad(\forall i \in \{1, \ldots, N\}). 
\]
\end{lem} 
\begin{proof} 
The first assertion follows from the implicit function theorem. 
The second assertion follows from Stokes'  theorem and 
$d\lambda_{\tau^0}(R_{\lambda_{\tau^0}}, \, \cdot ) = 0$. 
\end{proof} 

Now let us prove Lemma \ref{DerivativeOfCapacity}.
By spectrality of ECH spectral invariants,
there exists $C \in \mca{C}_\Z(Y, \lambda_{\tau^0})$ such that 
$c_\sigma(Y, \lambda_{\tau^0}) = C(\lambda_{\tau^0})$. 
Since $\lambda_{\tau^0}$ is nondegenerate, 
there are only finitely many such elements of $\mca{C}_\Z(Y, \lambda_{\tau^0})$, 
which we denote by $C^1, \ldots, C^J$. 
Then there exists an open neighborhood $U$ of $\tau^0$ in $[0,1]^N$ 
and $(C^1_\tau)_{\tau \in U}, \ldots, (C^J_\tau)_{\tau \in U}$ such that 
\begin{itemize} 
\item $C^j_\tau \in \mca{C}_\Z(Y, \lambda_\tau)$ for any $j \in \{1, \ldots, J\}$ and $\tau \in U$. 
\item $C^j_{\tau^0} = C^j$ for any $j \in \{1, \ldots, J\}$. 
\item $C^j_\tau$ depends smoothly on $\tau \in U$ for any $j \in \{1, \ldots, J\}$. 
\end{itemize} 

By spectrality, $C^0$-continuity, and the assumption that $\lambda_{\tau^0}$ is nondegenerate, 
we may assume 
\begin{equation}\label{181024_1} 
c_\sigma(Y, \lambda_\tau) \in \{C^1_\tau(\lambda_\tau), \ldots, C^J_\tau(\lambda_\tau)\} \qquad( \forall \tau \in U) 
\end{equation} 
by replacing $U$ with a smaller neighborhood if necessary. 

On the other hand, Lemma \ref{DerivativeOfAction} 
shows that, for any $j \in \{1, \ldots, J\}$, 
\[ 
U \to \R; \, \tau \mapsto C^j_\tau(\lambda_\tau)
\] 
is differentiable at $\tau^0$, and 
\[ 
\partial_{\tau_i} C^j_\tau(\lambda_\tau)(\tau^0) = C^j (\partial_{\tau_i} \lambda_\tau (\tau^0) )
\]
for every $i \in \{1, \ldots, N\}$. 
On the other hand, 
we assumed that 
$c_\sigma(Y, \lambda_\tau)$ is differentiable at 
$\tau^0$. 
Combined with (\ref{181024_1}), 
there exists $j \in \{1, \ldots, J\}$ such that 
\[ 
\partial_{\tau_i} C^j_\tau(\lambda_\tau)(\tau^0) = \partial_{\tau_i} c_\sigma(Y, \lambda_\tau)(\tau^0)
\] 
for every $i \in \{1, \ldots, N\}$. 
This completes the proof of Lemma \ref{DerivativeOfCapacity}.

\section{Periodic Reeb orbits on boundaries of star-shaped toric domains in $\C^2$}

In this section we study periodic Reeb orbits representing ECH homology classes 
of boundaries of star-shaped toric domains in $\C^2$. 
We prove a generic equidistribution result (Proposition \ref{190116_1}) for boundaries of strictly convex or concave toric domains, 
and state Question \ref{181226_1} (which is a toy model version of Question \ref{181218_1})
for general star-shaped toric domains. 

\subsection{Setting} 

$\gamma \subset (\R_{>0})^2$ is called a $C^\infty$ \textit{star-shaped} curve, 
if there exist $0 \le \theta_0 < \theta_1 < \pi/2$ and 
$\rho \in C^\infty([\theta_0, \theta_1], \R_{>0})$ such that 
\[ 
\gamma = \{ (\rho(\theta) \cos \theta, \rho(\theta) \sin \theta) \mid \theta_0 \le \theta \le \theta_1 \}.
\] 
For each $\theta$, we set $\gamma(\theta):= (\rho(\theta) \cos \theta, \rho(\theta) \sin \theta)$. 

For each $p \in \gamma$, 
let $\nu (p)$ denote the unit vector 
which is normal to $T_p \gamma$ and satisfies 
$p \cdot \nu(p)>0$. 
$\gamma$ is called \textit{strictly convex} if 
\[ 
\nu(p) \cdot (q-p) <0 
\] 
for any distinct points $p, q \in \gamma$. 
Also, $\gamma$ is called \textit{strictly concave} if 
\[ 
\nu(p) \cdot (q-p)  >0
\] 
for any distinct points $p, q\in \gamma$. 

$\gamma$ is called \textit{complete} if $\theta_0 = 0$ and $\theta_1 = \pi/2$. 
In the rest of this subsection, we assume that $\gamma$ is complete. 
Let us set 
\[
\mca{R}(\gamma): = \{ (\rho(0), 0), \, (0, \rho(\pi/2)) \} \cup \{ p \in \gamma  \mid \nu(p)  \in \R   \cdot \Q^2 \}. 
\]
For each $p=(x,y) \in \mca{R}(\gamma)$, 
we define $n(p) \in  \Z^2 \setminus \{(0,0)\}$ as follows: 
\begin{itemize} 
\item If $p=(\rho(0), 0)$, then $n(p):=(1,0)$. 
\item If $p = (0, \rho(\pi/2))$, then $n(p):=(0,1)$. 
\item Otherwise, $n(p)$ is characterized by the following two properties: 
(i): $n(p) = a \nu(p)$ for some $a \in \R_{>0}$, 
(ii): there does not exist a pair $(a, n')$ such that $a \in \Z_{\ge 2}$, $n' \in \Z^2$ and $an'=n(p)$. 
\end{itemize} 
We define $w(p) \in \R_{>0}$ by 
$w(p):= p \cdot n(p)$. 
We say that $\gamma$ is \textit{nice}, 
if $w(p_1), \ldots, w(p_k)$ are linearly independent over $\Q$ 
for any distinct elements $p_1, \ldots, p_k$ of $\mca{R}(\gamma)$. 

\begin{lem} 
In the set of all complete star-shaped curves, 
the set of nice curves is residual with respect to the $C^\infty$-topology.
\end{lem}
\begin{proof} 
Let 
$U:= \{ (x,y) \mid x^2 + y^2 = 1 \}$, 
and consider the map 
$\nu: \gamma \to U$ which maps each $p \in \gamma$ to $\nu(p) \in U$ which is defined above. 

For each $l \in \Z_{>0}$, let 
\[ 
U_l := \bigg\{ \bigg( \frac{i}{\sqrt{i^2 + j^2}} , \frac{j}{\sqrt{i^2+ j^2}} \bigg) \biggm{|}  (i,j) \in ([-l, l] \cap \Z)^2 \setminus \{ (0,0) \} \bigg\} \subset U, 
\] 
and consider the following three conditions for a $C^\infty$ star-shaped curve $\gamma$: 
\begin{itemize} 
\item 
$\nu(0, \rho(\pi/2)), \nu(\rho(0), 0) \notin U_l$. 
\item 
Every element of $U_l $ is a regular value of $\nu$. In particular, $\nu^{-1}(U_l)$ is a finite set. 
\item 
For any distinct points $p_1, \ldots, p_N \in \nu^{-1}(U_l ) \cup \{ (0, \rho(\pi/2)), (\rho(0), 0) \}$
and $a_1, \ldots, a_N \in ([-l, l] \cap \Z)^N \setminus \{ (0, \ldots, 0) \}$, 
there holds $a_1 w(p_1) + \cdots + a_N w(p_N) \ne 0$. 
\end{itemize}
Let $\mca{R}_l$ be the set of $\gamma$ such that
these conditions are satisfied. 
Then it is easy to check that $\mca{R}_l$ is open and dense for each $l$. 
Finally, if $\gamma \in \bigcap_l \mca{R}_l$, then $\gamma$ is nice. 
\end{proof} 

Let us define 
$\mu: \C^2 \to (\R_{\ge 0})^2$ by 
$\mu(z_1, z_2):= (\pi |z_1|^2, \pi |z_2|^2)$. 
Then, 
$Y_\gamma:= \mu^{-1}(\gamma)$ 
is a star-shaped hypersurface in $\C^2$. 
For each $i \in \{1, 2\}$, let us denote
\[
z_i  = x_i + \sqrt{-1} y_i =  \sqrt{{\mu_i}/{\pi}} \cdot  e^{\sqrt{-1} \theta_i},
\]
and define $\lambda \in \Omega^1(\C^2)$ by  
\[
\lambda: = \sum_{i=1}^2 \frac{x_i dy_i - y_i dx_i}{2}
=\sum_{i=1}^2 \frac{\mu_i d\theta_i}{2\pi}. 
\]
Then $\lambda_\gamma:= \lambda |_{Y_\gamma}$ is a contact form on $Y_\gamma$, and 
$\xi_\gamma:= \ker (\lambda_\gamma)$ (oriented by $d\lambda_\gamma>0$) satisfies 
\[
\ech(Y_\gamma,  \xi_\gamma, 0) = \bigoplus_{k=1}^\infty (\Z/2) \sigma_k
\]
where each $\sigma_k$ is homogeneous and 
$I(\sigma_{k+1}, \sigma_k)=2$ for every $k \in \Z_{\ge 1}$. 

For any $p \in \gamma$, 
let $\delta_p$ denote the distribution on $\gamma$ defined by 
\[
\delta_p(f):= f(p) \quad (\forall f \in C^\infty(\gamma, \R)).
\] 
For any $1$-dimensional current $C$ on $Y$ and $\omega \in \Omega^1(Y)$, 
let us define a distribution $C \wedge \omega$ on $Y$ by 
\[ 
(C \wedge \omega)(f) := C( \omega f) \quad ( \forall f \in C^\infty(Y)). 
\] 

\begin{lem}\label{190119_1} 
Let $k \in \Z_{>0}$, and let $C$ be a $1$-dimensional current on $Y_\gamma$ 
which represents $\sigma_k$ with $\lambda_\gamma$ (see Section 1.3). 
Then there exist $m_1, \ldots, m_N \in \Z_{>0}$ and 
distinct points $p_1, \ldots, p_N \in \mca{R}(\gamma)$ such that 
\[ 
(\mu |_{Y_\gamma})_*(C \wedge \lambda_\gamma) = \sum_{1 \le i  \le N} m_i w(p_i) \delta_{p_i}. 
\] 
In particular, 
\begin{equation}\label{190119_2}
c_{\sigma_k}(Y_\gamma, \lambda_\gamma) = \sum_{1 \le i \le N} m_i w(p_i). 
\end{equation} 
If $\gamma$ is nice, then the set of pairs $\{ (m_i, p_i) \}_{1 \le i \le N}$ is characterized by (\ref{190119_2}). 
\end{lem} 
\begin{proof}
Let us denote $C= \sum_{1 \le i \le N} m_i \gamma_i$, 
where $m_1, \ldots, m_N$ are positive integers, and 
$\gamma_1, \ldots, \gamma_N$ are distinct elements of $\mca{P}(Y_\gamma, \lambda_\gamma)$. 
For each $i$, let $p_i:= \mu(\gamma_i)$. Then $p_i \in \mca{R}(\gamma)$ and $T_{\gamma_i} =  w(p_i)$. 
Then we obtain 
\[ 
(\mu|_{Y_\gamma})_*(C \wedge \lambda_\gamma) = \sum_{1 \le i \le N} m_i  \cdot (\mu |_{Y_\gamma})_*(\gamma_i \wedge \lambda_\gamma) = \sum_{1 \le i \le N}  m_i w(p_i) \delta_{p_i}. 
\] 
By evaluating the constant function $1$, we obtain (\ref{190119_2}). 
The last assertion is obvious from the definition of niceness. 
\end{proof}

When $\gamma$ is nice, 
the distribution $\sum_{1 \le i  \le N} m_i  w(p_i) \delta_{p_i}$ in Lemma \ref{190119_1} 
is uniquely determined, i.e. not depend on choices of $C$; let us denote it by $\mathfrak{D}_k$. 
On the other hand, we define a distribution $\mathfrak{D}_\vol$ on $\gamma$ by 
$\mathfrak{D}_\vol:= (\mu |_{Y_\gamma})_*(\lambda_\gamma \wedge d \lambda_\gamma)$. 
Let us consider the following condition: 
\begin{equation}\label{181220_1}
\lim_{k \to \infty}  \frac{\mathfrak{D}_k(f)}{\sqrt{2k}} = \frac{\mathfrak{D}_\vol(f)}{\sqrt{\vol (Y_\gamma, \lambda_\gamma)}}
\qquad
(\forall f \in C^\infty(\gamma, \R)). 
\end{equation}

\begin{rem}\label{190123_1} 
(\ref{181220_1}) is equivalent to the condition that 
any (not necessarily complete) star-shaped curve $\gamma' \subset \gamma$ satisfies 
\[ 
\lim_{k \to \infty} \frac{\mathfrak{D}_k(\gamma')}{\sqrt{2k}}  = \frac{\mathfrak{D}_\vol(\gamma')}{\sqrt{\vol (Y_\gamma, \lambda_\gamma)}}. 
\] 
Here, for any distribution $\mathfrak{D}$ on $\gamma$, we define 
\[ 
\mathfrak{D}(\gamma'):= \inf\{ \mathfrak{D}(f) \mid f \in C^\infty(\gamma, \R_{\ge 0}), \, f|_{\gamma'} \equiv 1 \}. 
\] 
\end{rem} 

The goal of the rest of this paper is to prove the following result: 

\begin{prop}\label{190116_1} 
(\ref{181220_1}) holds if $\gamma$ is nice and strictly convex or concave. 
\end{prop} 

Based upon this result, let us propose the following question, 
which is a toy model version of Question \ref{181218_1}: 

\begin{que}\label{181226_1} 
Does 
(\ref{181220_1}) hold for a 
$C^\infty$-generic nice curve $\gamma$ ? 
\end{que}

To prove Proposition \ref{190116_1}, 
in the next subsection
we state and prove versions of isoperimetric inequality. 

\subsection{Versions of isoperimetric inequality} 

$\gamma \subset (\R_{>0})^2$ is called a \textit{continuous star-shaped curve}, 
if there exist $0 \le \theta_0 < \theta_1 \le \pi/2$
and $\rho \in C^0([\theta_0, \theta_1], \R_{>0})$ such that 
\[
\gamma = \{ (\rho(\theta) \cos \theta, \rho(\theta) \sin \theta ) \mid \theta \in [\theta_0, \theta_1] \}. 
\] 
For each $\theta$, we denote $\gamma(\theta):= (\rho(\theta) \cos \theta, \rho(\theta) \sin \theta )$. 
$\gamma$ is called \textit{convex} if 
\[ 
\det (\gamma(\theta') - \gamma(\theta), \gamma(\theta'') - \gamma(\theta')) \ge 0
\] 
for any $\theta < \theta' < \theta''$. 
$\gamma$ is called \textit{concave} if 
\[ 
\det (\gamma(\theta') - \gamma(\theta), \gamma(\theta'') - \gamma(\theta')) \le 0
\] 
for any $\theta < \theta ' < \theta''$. 

$\gamma$ is called \textit{complete} 
if $\theta_0 = 0$ and $\theta_1 = \pi/2$. 
When $\gamma$ is complete, 
let $A(\gamma)$ denote the area of the region 
$\{ (r\cos \theta, r \sin \theta) \mid 0 \le \theta \le \pi/2, \, 0 \le r \le \rho(\theta) \}$. 

Let $\gamma_0$ and $\gamma$ be continuous convex curves, and 
suppose that $\gamma_0$ is complete. 
For all but countably many $p \in \gamma$, 
there exists a tangent line of $\gamma$ at $p$. 
Hence one can define $\nu(p)$ in the same manner as in the $C^\infty$-case. 
Now we define 
\[ 
A_{\gamma_0}(\gamma):= \int_\gamma (\max_{q \in \gamma_0}  q \cdot \nu(p) ) \, d \mu_\gamma(p),  
\] 
where the measure $\mu_\gamma$ is defined by 
$d\mu_\gamma = |\partial_\theta \gamma| d \theta$. 
It is easy to check that $A_{\gamma_0}$ is continuous with respect to the Hausdorff distance: 

\begin{lem}\label{181222_1} 
Let $\gamma_0$ be a complete continuous convex curve, and 
$(\gamma_j)_{j \ge 1}$ be a sequence of (not necessarily complete) continuous convex curves which converges to a continuous convex curve $\gamma_\infty$ with respect to the Hausdorff distance. 
Then 
$\lim_{j \to \infty} A_{\gamma_0}(\gamma_j) = A_{\gamma_0}(\gamma_\infty)$. 
\end{lem}

Now let us state a version of isoperimetric inequality in this  setting. 

\begin{lem}\label{181222_2} 
Let $\gamma_0$ and $\gamma$ be continous, convex, complete curves. 
Then there holds 
\[ 
\frac{A_{\gamma_0}(\gamma)}{\sqrt{A(\gamma)}} \ge \frac{A_{\gamma_0}(\gamma_0)}{\sqrt{A(\gamma_0)}}. 
\] 
The equality holds if and only if $\gamma = c \gamma_0$ for some $c \in \R_{>0}$. 
\end{lem} 

The proof of Lemma \ref{181222_2} is omitted, since it is similar to 
the proof of the generalized isoperimetric inequality in \cite{isoperimetric}. 

We also need a similar result for concave curves. 
Let $\gamma_0$ and $\gamma$ be continuous concave curves, 
and suppose that $\gamma_0$ is complete. 
Now we define 
\[ 
A_{\gamma_0}(\gamma):= \int_\gamma (\min_{q \in \gamma_0}  q \cdot \nu(p) ) \, d \mu_\gamma(p). 
\] 
$A_{\gamma_0}$ is continuous with respect to the Hausdorff distance. 
Now let us state a version of isoperimetric inequality in this setting. 

\begin{lem}\label{190123_2} 
Let $\gamma_0$ and $\gamma$ be continuous, concave, complete curves. 
Then there holds 
\[ 
\frac{A_{\gamma_0}(\gamma)}{\sqrt{A(\gamma)}} \le \frac{A_{\gamma_0}(\gamma_0)}{\sqrt{A(\gamma_0)}}. 
\] 
The equality holds if and only if $\gamma = c \gamma_0$ for some $c \in \R_{>0}$. 
\end{lem}
\begin{proof}
After scaling, we may assume that 
$A(\gamma) = A(\gamma_0)$, and there holds
\begin{align*} 
&\gamma \cap (\R_{>0} \times \{ 0 \} ) = \{  (a, 0) \}, \quad
\gamma_0 \cap (\R_{>0} \times \{ 0 \} ) = \{ (a_0, 0) \}, \\
&\gamma \cap  ( \{ 0 \} \times \R_{>0})   = \{ (0, b)\}, \quad
\gamma_0 \cap ( \{ 0 \} \times \R_{>0}) = \{ (0, b_0) \} 
\end{align*} 
where $a, a_0, b, b_0 \in  (0, 1)$. 
Let us define continuous convex curves 
\begin{align*} 
\bar{\gamma}&:=  \{ (x,y) \mid (1-x, 1-y)  \in \gamma \cup (a,1] \times \{ 0 \}  \cup \{ 0 \} \times  (b, 1] \}, \\
\bar{\gamma}_0&:= \{ (x,y) \mid (1-x, 1-y) \in \gamma_0 \cup (a_0, 1] \times \{ 0 \} \cup \{ 0 \}  \times (b_0, 1] \}. 
\end{align*} 
Then $A(\bar{\gamma}) = A(\bar{\gamma}_0) = 1 - A(\gamma)$, and 
$A_{\bar{\gamma}_0}(\bar{\gamma}) + A_{\gamma_0}(\gamma) = 2$. 
Then Lemma \ref{181222_2} implies 
\[ 
A_{\gamma_0}(\gamma) = 2 - A_{\bar{\gamma}_0}(\bar{\gamma}) \le 2  - A_{\bar{\gamma}_0}(\bar{\gamma}_0) = A_{\gamma_0}(\gamma_0). 
\] 
The equality holds if and only if $\bar{\gamma}_0 = \bar{\gamma}$, which is equivalent to $\gamma_0 = \gamma$. 
\end{proof} 

\subsection{Proof of Proposition \ref{190116_1}}

In this section we verify (\ref{181220_1})
when $\gamma$ is a complete star-shaped curve which is nice and either strictly convex or concave. 
Throughout this section we assume $A(\gamma)=1$, without loss of generality. 

\subsubsection{When $\gamma$ is nice and strictly convex} 

For any convex integral path $\Lambda$ (see \cite{CCG} Definition A.2, where it is called convex lattice path), 
let $L(\Lambda)$ denote the number of lattice points on the region bounded by $\Lambda$ and the $x$ and $y$ axes. 
By \cite{CCG} (see also \cite{Beyond}), 
for every $k \in \Z_{>0}$ there holds 
\begin{equation}\label{190122_1} 
c_{\sigma_k}(Y_\gamma, \lambda_\gamma) = 
\min \{ A_\gamma(\Lambda) \mid L(\Lambda) = k \}. 
\end{equation} 

For any convex integral path $\Lambda$, 
we say that $e \subset \Lambda$ is an $\textit{edge}$ of $\Lambda$
if $e$ is a line segment of positive length, 
the boundary points of $e$ are lattice points, 
and there exists no other lattice point on $e$. 
Let $E(\Lambda)$ denote the set of all edges of $\Lambda$. 
For any $e \in E(\Lambda)$, let $\nu(e)$ denote the unit vector which is normal to $e$ and points outwards. 

For each $k \in \Z_{>0}$, 
there exists a convex integral path $\Lambda_k$ 
which satisfies the following conditions: 
\begin{itemize}
\item[(i):] $L(\Lambda_k)  \ge k$. 
\item[(ii):] $A_\gamma(\Lambda_k) = c_{\sigma_k}(Y_\gamma, \lambda_\gamma)$. 
\item[(iii):] Every edge $e$ of $\Lambda_k$ satisfies $\nu(e) \in \{ (\cos \theta, \sin \theta) \mid \ep - \pi/2 \le \theta \le \pi - \ep \}$
, where $\ep$ is a positive constant which depends only on $\gamma$. 
\end{itemize} 
We fix a sequence $(\Lambda_k)_k$ which satisfies this condition until the end of the proof. 
The next lemma is the key observation in the proof. 

\begin{lem}\label{190228_1} 
$\lim_{k \to \infty} \frac{\Lambda_k}{\sqrt{k}} = \gamma$ with respect to the Hausdorff distance.
\end{lem}
\begin{proof} 
We first prove $\sup_k \diam(\Lambda_k/ \sqrt{k}) < \infty$. 
If this is not true, there exists an increasing sequence of positive integers, 
which we denote by $(k_j)_j$, such that 
$\lim_{j \to \infty} \diam(\Lambda_{k_j}/\sqrt{k_j}) = \infty$. 
Then
$\lim_{j \to \infty} A_\gamma(\Lambda_{k_j})/ \sqrt{k_j} = \infty$, 
contradicting the fact that 
$c_{\sigma_k}(Y_\gamma, \lambda_\gamma)$ 
is of order $\sqrt{k}$. 

If the lemma does not hold, then 
there exists an increasing sequence of positive integers, 
which we denote by $(k_j)_j$, and a continuous convex curve $\gamma' \ne \gamma$
such that 
$\lim_{j \to \infty} \Lambda_{k_j}/ \sqrt{k_j} = \gamma'$. 
By $L(\Lambda_{k_j}) \ge k_j$ for every $j$, 
we obtain 
$A(\gamma') \ge 1 = A(\gamma)$. 

By $\gamma' \ne \gamma$ and $A(\gamma') \ge A(\gamma)$, 
Lemma \ref{181222_2}
implies 
$A_\gamma(\gamma') > A_\gamma(\gamma)$. 
On the other hand, 
there exists a sequence $(\Lambda'_j)_j$ of convex integral paths 
such that the following conditions hold:
\begin{itemize} 
\item 
$L(\Lambda'_j)= k_j$ for every $j$. 
\item 
$\lim_{j \to \infty} \Lambda'_j/\sqrt{k_j} = \gamma$ with respect to the Hausdorff distance. 
\end{itemize} 
Then there holds 
\[ 
\lim_{j \to \infty} A_\gamma(\Lambda'_j)/\sqrt{k_j} = A_\gamma(\gamma)  < A_\gamma(\gamma') = \lim_{j \to \infty} A_\gamma(\Lambda_{k_j})/\sqrt{k_j}. 
\] 
Hence $A_\gamma(\Lambda'_j) < A_\gamma(\Lambda_{k_j})  = c_{\sigma_{k_j}}(Y_\gamma, \lambda_\gamma)$ for sufficiently large $j$,
which contradicts (\ref{190122_1}). 
\end{proof} 

Since $\gamma$ is strictly convex, 
for every $e \in E(\Lambda_k)$,
there exists unique $p(e) \in \gamma$ such that 
\[ 
\nu(e) \cdot p(e) = \max_{p \in \gamma} \nu(e) \cdot p. 
\] 
Let us define $m(e) \in \Z_{>0}$ as follows: 
\begin{itemize} 
\item If  $p(e) \notin \{ (\rho(0), 0), ( 0, \rho(\pi/2)) \}$, then $m(e)=1$. 
\item If $p(e) = (\rho(0), 0)$, then $m(e)$ is the absolute value of the $y$ -component of $e$. 
\item If $p(e) = (0, \rho(\pi/2))$, then $m(e)$ is the absolute value of the $x$ -component of $e$. 
\end{itemize} 

Then we obtain 
\[
A_\gamma(\Lambda_k) = \sum_{e \in E(\Lambda_k)}  m(e) w(p(e)).
\]
Since $\gamma$ is nice, 
Lemma \ref{190119_1} implies 
\[ 
\mathfrak{D}_k = \sum_{e \in E(\Lambda_k)}   m(e) w(p(e))  \delta_{p(e)}. 
\] 
By Remark \ref{190123_1},
to verify (\ref{181220_1}) 
 it is sufficient to prove 
\[ 
\lim_{k \to \infty} \frac{\mathfrak{D}_k(\gamma')}{\sqrt{2k}}  = \frac{\mathfrak{D}_\vol(\gamma')}{\sqrt{\vol (Y_\gamma, \lambda_\gamma)}}
\] 
for every (not necessarily complete) star-shaped curve $\gamma' \subset \gamma$. 
For each $k$, 
\[ 
\Lambda'_k:= \bigcup_{\substack{ e \in E(\Lambda_k) \\ p(e) \in \gamma'}}  e
\] 
is a continuous convex curve which satisfies $\mathfrak{D}_k(\gamma') = A_\gamma(\Lambda'_k)$. 
The next lemma easily follows from Lemma \ref{190228_1} and the assumption that $\gamma$ is strictly convex: 

\begin{lem}\label{190122_2} 
$\lim_{k \to \infty} \Lambda'_k / \sqrt{k} = \gamma'$
with respect to the Hausdorff distance. 
\end{lem} 

Now one can complete the proof by 
\[
\lim_{k \to \infty} \frac{\mathfrak{D}_k(\gamma')}{\sqrt{2k}} 
= 
\lim_{k \to \infty} \frac{A_\gamma(\Lambda'_k)}{\sqrt{2k}} 
= 
\frac{A_\gamma(\gamma')}{\sqrt{2}}
= 
\frac{\mathfrak{D}_\vol(\gamma')}{\sqrt{\vol(Y_\gamma, \lambda_\gamma)}}. 
\] 
The second equality follows from Lemma \ref{190122_2} and continuity of $A_\gamma$ 
with respect to the Hausdorff distance. 
The last equality holds since 
$A_\gamma(\gamma') = \mathfrak{D}_\vol(\gamma')$ 
can be checked by direct computations, and 
$\vol(Y_\gamma, \lambda_\gamma) = \mathfrak{D}_\vol(\gamma) = A_\gamma(\gamma) = 2 A(\gamma) = 2$. 

\subsubsection{When $\gamma$ is nice and strictly concave} 

Let us sketch the proof of (\ref{181220_1})
when $\gamma$ is nice and strictly concave. 
For any concave integral path $\Lambda$ (see Definition 1.18 in \cite{CCGetal}), 
let $L'(\Lambda)$ denote the number of lattice points on the region bounded by $\Lambda$ and the $x$ and $y$ axes, 
not including lattice points on $\Lambda$. 
There holds 
\begin{equation}\label{190307_1} 
c_{\sigma_k}(Y_\gamma, \lambda_\gamma) = \max \{ A_\gamma(\Lambda) \mid L'(\Lambda) = k-1 \} 
\end{equation}
for any $k \in \Z_{>0}$, by \cite{CCGetal} Theorem 1.21. 

Then there exists a sequence 
$(\Lambda_k)_k$ of concave integral paths satisfying the following conditions for every $k$: 
\begin{itemize} 
\item[(i):] $L'(\Lambda_k) \le k-1$. 
\item[(ii):] $c_{\sigma_k}(Y_\gamma, \lambda_\gamma) = A_\gamma(\Lambda_k)$. 
\item[(iii):] For every edge $e$ of $\Lambda_k$, 
\[
\nu(e) \in \{ ( \cos \theta, \sin \theta) \mid \ep \le \theta \le \pi/2- \ep\}, 
\] 
where $\ep$ is a positive constant which depends only on $\gamma$. 
\end{itemize} 

One can prove $\sup_k \diam (\Lambda_k /\sqrt{k} ) < \infty$ from assumptions (ii), (iii) and $c_{\sigma_k}(Y_\gamma, \lambda_\gamma)$ is of order $\sqrt{k}$. 
Now let us prove $\lim_{k \to \infty} \Lambda_k/ \sqrt{k} = \gamma$. 
If this is not the case, 
there exist an increasing sequence of positive integers $(k_j)_j$
and a continuous concave curve $\gamma' \ne \gamma$ 
such that $\lim_{j \to \infty} \Lambda_{k_j}/\sqrt{k_j} = \gamma'$. 
By $L'(\Lambda_{k_j}) \le k_j -1$ for every $j$, 
we obtain $A(\gamma') \le 1 = A(\gamma)$. 
Since $\gamma' \ne \gamma$ and $A(\gamma') \le A(\gamma)$, 
Lemma \ref{190123_2} implies 
$A_\gamma(\gamma') < A_\gamma(\gamma)$. 
On the other hand, 
there exists a sequence $(\Lambda'_j)_j$ 
of concave integral paths 
such that the following conditions hold: 
\begin{itemize} 
\item 
$L'(\Lambda'_j) = k_j - 1$ for every $j$. 
\item
$\lim_{j \to \infty} \Lambda'_j  = \gamma$, thus 
$\lim_{j \to \infty} A_\gamma(\Lambda'_j) /\sqrt{k_j} = A_\gamma(\gamma)$. 
\end{itemize} 
Then $A_\gamma(\Lambda'_j) > A_\gamma(\Lambda_{k_j}) = c_{\sigma_{k_j}}(Y_\gamma, \lambda_\gamma)$
for sufficiently large $j$, which contradicts (\ref{190307_1}). 
Thus we have shown 
$\lim_{k \to \infty} \Lambda_k/\sqrt{k} = \gamma$. 
The rest of the proof is similar to the convex case, 
and details are omitted.

\end{document}